\documentclass[11pt]{amsart}
\usepackage[a4paper,margin=2.5cm]{geometry}
\usepackage{amsmath,amsthm,amssymb,amscd,amsfonts,color}
\usepackage{algorithm}
\usepackage{algpseudocode}
\usepackage{dynkin-diagrams}
\usepackage{hyperref}

\theoremstyle{plain}
\newtheorem{theorem}{Theorem}[section]
\newtheorem{proposition}[theorem]{Proposition}
\newtheorem{lemma}[theorem]{Lemma}

\numberwithin{equation}{section}

\theoremstyle{definition}

\newtheorem{remark}[theorem]{Remark}

\newcommand{\cC}{\mathcal{C}}

\newcommand{\Q}{\mathbb{Q}}
\newcommand{\R}{\mathbb{R}}
\newcommand{\Z}{\mathbb{Z}}

\makeatletter
\renewcommand{\fnum@algorithm}{\fname@algorithm}
\makeatother

\begin{document}

\title[Real toric varieties of types $E_7$ and $E_8$]{The Betti numbers of real toric varieties associated to Weyl chambers of types $E_7$ and $E_8$}

\author[S. Choi]{Suyoung Choi}
\address{Department of mathematics, Ajou University, 206, World cup-ro, Yeongtong-gu, Suwon 16499,  Republic of Korea}
\email{schoi@ajou.ac.kr}

\author[Y. Yoon]{Younghan Yoon}
\address{Department of mathematics, Ajou University, 206, World cup-ro, Yeongtong-gu, Suwon 16499,  Republic of Korea}
\email{younghan300@ajou.ac.kr}

\author[S. Yu]{Seonghyeon Yu}
\address{Department of mathematics, Ajou University, 206, World cup-ro, Yeongtong-gu, Suwon 16499,  Republic of Korea}
\email{yoosh0319@ajou.ac.kr}

\date{\today}

\thanks{The authors were supported by the National Research Foundation of Korea Grant funded by the Korean Government (NRF-2019R1A2C2010989).}


\begin{abstract}
We compute the rational Betti numbers of the real toric varieties associated to Weyl chambers of types $E_7$ and $E_8$, completing the computations for all types of root systems.
\end{abstract}

\maketitle

\section{Introduction}
A root system is a finite set of vectors in a finite dimensional Euclidean space that is closed under the action of a Weyl group \cite{Hall_book}.
It is known \cite{Procesi1990} that a root system of type $R$ generates a non-singular complete fan $\Sigma_R$ by its Weyl chambers and co-weight lattice and that $\Sigma_R$ corresponds to a smooth compact (complex) toric variety $X_R$ by the fundamental theorem for toric geometry.
In particular, the real locus of $X_R$ is called \emph{the real toric variety associated to the Weyl chambers}, denoted by $X^\R_R$.

It is natural to ask for the topological invariants of $X^\R_R$.
By \cite{Davis-Januszkiewicz1991}, the $\Z_2$-Betti numbers of $X^\R_R$ can be completely computed from the face numbers of $\Sigma_R$.
In general, however, computing the rational Betti numbers of a real toric variety is much more difficult.
In 2012, Henderson \cite{Henderson2012} computed the rational Betti numbers of $X^\R_{A_n}$.
The computation of other classic and exceptional types has been carried out using the formulae for rational Betti numbers developed in  \cite{ST2012} or \cite{Choi-Park2017_torsion}.
At the time of writing this paper, results have been established for $X_R^\R$ of all types except $E_7$ and $E_8$.

For the classical types $R = A_n, B_n, C_n,$ and $D_n$, the $k$th Betti numbers $\beta_k$ of $X^{\R}_{R}$ are known to be as follows (see \cite{Choi-Kaji-Park2019}, \cite{Choi-Park-Park2017}, \cite{Henderson2012}):
\begin{align*}
    \beta_k(X^{\R}_{A_n};\Q) &= {n+1 \choose 2k}a_{2k}, \\
    \beta_k(X^{\R}_{B_n};\Q) &= {n \choose 2k}b_{2k} + {n \choose 2k-1}b_{2k-1}, \\
    \beta_k(X^{\R}_{C_n};\Q) &= {n \choose 2k-2}\left(2^n -  2^{2k-2} \right)a_{2k-2} + {n \choose 2k}(2b_{2k}-2^{2k}a_{2k}), \text{ and } \\
    \beta_k(X^{\R}_{D_n};\Q) &= {n \choose 2k-4}\left(2^{2k-4} + (n-2k+2)2^{n-1} \right) a_{2k-4} + {n \choose 2k}(2b_{2k}-2^{2k}a_{2k}),
\end{align*}
where $a_r$ is the $r$th Euler zigzag number (A000111 in \cite{oeis}) and $b_r$ is the $r$th generalized Euler number (A001586 in \cite{oeis}).

For the exceptional types $R = G_2, F_4$, and $E_6$, the Betti numbers of $X^{\R}_{R}$ are as in Table~\ref{EFG} (see \cite[Proposition~3.3]{Cho-Choi-Kaji2019}).
    \begin{table}[h]
    \centering
    \begin{tabular}{c|c|c|c}
     \hline
     $\beta_{k}(X^{\R}_{R})$ & $R = G_2$ & $R = F_4$ & $R = E_6$ \\
     \hline
     $k = 0$  & $1$ & $1$ & $1$  \\
     \hline
     $k = 1$  & $9$ & $57$ & $36$ \\
     \hline
     $k = 2$ & $0$ & $264$ & $1{,}323$ \\
     \hline
     $k = 3$ & $0$ & $0$ & $4{,}392$ \\
     \hline
    \end{tabular}\\
    \caption{Nonzero Betti numbers of $X^{\R}_{G_2}$, $X^{\R}_{F_4}$, and $X^{\R}_{E_6}$}
    \label{EFG}
    \end{table}

The purpose of this paper is to compute the Betti numbers for the remaining exceptional types $E_7$ and $E_8$.
The reason why these cases have so far remained unsolved is that, as remarked in \cite{Cho-Choi-Kaji2019}, the corresponding fans are too large to be dealt with.
We provide a technical method to decompose the Coxeter complex; using this method, we obtain explicit subcomplexes $K_S$ that play an important role in our main computation.
Furthermore, we obtain a smaller simplicial complex by removing vertices in $K_S$ without changing its homology groups, so that the Betti numbers can be computed.

\begin{theorem}
    The $k$th Betti numbers $\beta_{k}$ of $X^{\R}_{E_7}$ and $X^{\R}_{E_8}$ are as follows.

     $$\beta_{k}(X^{\R}_{E_7};\Q)= \begin{cases}
    	1, & \mbox{if} \ k=0\\
    	63, & \mbox{if} \ k=1\\
    	8{,}127, & \mbox{if} \ k=2\\
    	131{,}041, & \mbox{if} \ k=3\\
    	122{,}976, & \mbox{if} \ k=4\\
    	0, & \mbox{otherwise.}
    \end{cases}$$\\

$$\beta_{k}(X^{\R}_{E_8};\Q)= \begin{cases}
	1, & \mbox{if} \ k=0\\
	120, & \mbox{if} \ k=1\\
	103{,}815, & \mbox{if} \ k=2\\
	6{,}925{,}200, & \mbox{if} \ k=3\\
	23{,}932{,}800, & \mbox{if} \ k=4\\
	0, & \mbox{otherwise.}
\end{cases} $$
\end{theorem}

\section{Real toric varieties associated to the Weyl chambers} \label{sec:Real toric varieties associated to the Weyl chambers}

We recall some known facts about the real toric varieties associated to the Weyl chambers, following the notation in \cite{Cho-Choi-Kaji2019} unless otherwise specified.

Let $\Phi_R$ be an irreducible root system of type $R$ in a finite dimensional Euclidean space $E$ and $W_R$ its Weyl group.
Then the reflections, namely the elements of $W_R$, give connected components in $E$, called the \emph{Weyl chambers}.
We fix a particular Weyl chamber, called the \emph{fundamental Weyl chamber} $\Omega$; its rays $\omega_1, \ldots ,\omega_n$ are called \emph{the fundamental co-weights}.
Then, $\Z(\{\omega_1,\ldots,\omega_n\})$ has a lattice structure and is called the co-weight lattice.
Consider the set of Weyl chambers as a nonsingular complete fan $\Sigma_R$ with the co-weight lattice.
From the set $V=\{v_1, \ldots,v_m\}$ of rays spanning $\Sigma_R$ we obtain the simplicial complex $K_R$, called \emph{the Coxeter complex} of type $R$ on $V$, whose faces in $K_R$ are obtained via the corresponding faces in $\Sigma_R$ (see \cite{Bjorner1984} for more details).
The directions of rays on the co-weight lattice give a linear map $\lambda_R \colon V \to \Z^n$.
In addition, the composition map $\Lambda_R \colon V \overset{\lambda_R}{\to} \Z^n \overset{\text{mod}~2}{\longrightarrow} \Z_2^n$ can be expressed as an $n \times m$ (mod $2$) matrix, called a (mod $2$) \emph{characteristic matrix}.
Let $S$ be an element of the row space $Row(\Lambda_R)$ of $\Lambda_R$.
Since each column of $\Lambda_R$ corresponds to a vertex $v \in V$, $S$ can be regarded as a subset of $V$.
Let us consider the induced subcomplex $K_S$ of $K_R$ with respect to $S$.
It is known that the reduced Betti numbers of $K_{S}$ deeply correspond to the Betti numbers of $X^\R_R$.

\begin{theorem}\cite{Cho-Choi-Kaji2019}\label{geomain}
For any root system $\Phi_{R}$ of type $R$, let $W_{R}$ be the Weyl group of $\Phi_{R}$. Then, there is a $W_{R}$-module isomorphism
$$H_{\ast}(X_{R}^{\R}) \cong \bigoplus_{S\in Row(\Lambda_{R})} \widetilde{H}_{\ast-1}(K_{S}),$$
where $K_{S}$ is the induced subcomplex of $K_{R}$ with respect to $S$.
\end{theorem}

Since, by Theorem~\ref{geomain}, $K_S \cong K_{gS}$ for $S \in Row(\Lambda_R)$ and $g\in W_R$, we need only investigate representatives $K_S$ of the $W_R$-orbits in $Row(\Lambda_R)$.

\begin{proposition}\cite{Cho-Choi-Kaji2019}\label{E7E8orbit}
For type $E_7$, there are $127$ nonzero elements in Row($\Lambda_{E_7}$).
In addition, there are exactly three orbits (whose representatives are denoted by $S_1,S_2$, and $S_3$), and the numbers of elements for each orbit are $63, 63,$ and $1$, respectively.

For type $E_8$, there are $255$ nonzero elements in Row($\Lambda_{E_8}$). There are only two orbits (whose representatives are denoted by $S_4$ and $S_5$), and the numbers of elements for each orbits are $120$ and $135$, respectively.
\end{proposition}

Thus, for our purpose, it is enough to compute the (reduced) Betti numbers of $K_{S_i}$ for $1 \leq i \leq 5$.
For practical reasons such as memory errors and large time complexity, it is not easy to obtain $K_S$ directly by computer programs.
The remainder of this section is devoted to introducing an effective way to obtain $K_{S}$.

For a fixed fundamental co-weight $\omega$, let $H_\omega$ be the isotropy subgroup of $W_R$ to $\omega$.

\begin{lemma}\label{cosetdecom}
	For type $R$, let $K_\omega$ be a subcomplex of $K_R$ induced by the set $\{g \cdot \Omega \mid g \in H_\omega \}$, where $\Omega$ is the fundamental Weyl chamber.
	Then there is a decomposition of the Coxeter complex $K_R$ as follows:
	$$K_R = \bigsqcup_{g \in {W_R}/{H_\omega}}  {K^g},$$ where $K^g = g \cdot K_\omega$.
\end{lemma}

\begin{proof}
	For any maximal simplex $\sigma \in K_R$, there exists a unique $h \in W_R$ such that $h \cdot \Omega = \sigma$ by Propositions~8.23 and 8.27 in \cite{Hall_book}\label{free}.
    It follows that $h$ is uniquely contained in $g \cdot H_\omega$ for some $g \in W_R$. Thus, $\sigma = h \cdot \Omega$ is a maximal simplex of $K^g$, and all $K^g$s are pairwise disjoint.
\end{proof}

By the above lemma, $K_S$ is also decomposed into $K^g_S := K^g \cap K_S$ for all coset representations $g \in {W_R}/{H_\omega}$.
The set of all maximal simplices of $K_S$ is then obtainable as the union of the sets of all maximal simplices of $K^g_S$ for all $g \in {W_R}/{H_\omega}$.
However, for types $E_7$ and $E_8$, since $K^g$ still has many facets, it is not easy to obtain $K^g_S$ from $K^g$ directly; see Table~\ref{K_R}.

\begin{table}[H]
	\renewcommand{\arraystretch}{1.3}
	\centering
	\begin{tabular}{c|c|c}
		\hline
		 & $R=E_7$ & $R=E_8$ \\
		\hline
		$\#$ vertices of $K_R$ & 17{,}642 & 881{,}760 \\
		\hline
		$\#$ chambers of $K_R$ & 2{,}903{,}040 & 696{,}729{,}600 \\
		\hline
		$| W_R /H_\omega|$ & 126 & 240 \\
		\hline
		$\#$ chambers of $K^g$ & 23{,}040 & 2{,}903{,}040 \\
		\hline
	\end{tabular}\\
	\caption{Statistics for $K_R$ when $R=E_7$ and $E_8$}
	\label{K_R}
\end{table}

We establish a lemma to improve the time complexity.
Denote by $V^g_S$ the set of vertices in $K^g_S$.
\begin{lemma}\label{criterion}
    Let $g,h \in W_{R}/{H_\omega}$.
    If $g \cdot V^{h}_S = V^{gh}_S$, then $g \cdot K^{h}_S = K^{gh}_S$.
\end{lemma}

\begin{proof}
    For $g \in {W_R}/{H_\omega}$, we naturally consider $g$ as a simplicial isomorphism from $K^h$ to $K^{gh}$.
    If $g\cdot V^h_S = V^{gh}_S$, then the restriction of $g$ to $K^h_S$ is well-defined. Thus, $g$ is also regarded as a simplicial isomorphism between $K^h_S$ and $K^{gh}_S$.
\end{proof}

By the above lemma, in the case when $g\cdot V^{h}_S = V^{gh}_S$, $K^{gh}_S$ is obtainable without any computation.
Since checking the hypothesis of the lemma is much easier than forming $K^g_S$ from $K^g$, a good deal of time can be saved.
Using this method, one can obtain $K_S$ within a reasonable time with standard computer hardware.

\section{Simplicial complexes for types $E_7$ and $E_8$}

Since each $K_S$ for the types $E_7$ or $E_8$ is too large for direct computation, it is impossible using existing methods to compute their Betti numbers directly.
In this section, we introduce the specific smaller simplicial complex $\widehat{K}_S$ whose homology group is isomorphic as a group to that of $K_S$.

Let $K$ be a simplicial complex.
The \emph{link} $Lk_{K}(v)$ of $v$ in $K$ is a set of all faces $\sigma \in K$ such that $v \notin \sigma$ and $\{v\} \cup \sigma \in K$, while the (closed) \emph{star} $St_{K}(v)$ of $v$ in $K$ is a set of all faces $\sigma \in K$ such that $\{v\} \cup \sigma \in K$.
For a vertex $v$ of $K_S$ satisfying $Lk_{K}(v) \neq \emptyset$, we consider the following Mayer-Vietoris sequence:
$$
    \cdots \rightarrow \widetilde{H}_{k}(Lk_{K}(v)) \rightarrow \widetilde{H}_{k}(K-v)\oplus \widetilde{H}_{k}(St_{K}(v)) \rightarrow \widetilde{H}_{k}(K) \rightarrow \widetilde{H}_{k-1}(Lk_{K}(v)) \rightarrow \cdots,
$$
where $K-v = \{\sigma-\{v\} \mid \sigma \in K\}$ and $k$ is a positive integer.
We note that $\widetilde{H}_{k}(St_{K}(v)) = 0$ for $k \geq 0$ since $St_{K}(v)$ is a topological cone.
Therefore, for $k \geq 0$, if $\widetilde{H}_{k}(Lk_{K}(v))$ is trivial, then $\widetilde{H}_{k}(K-v) \cong \widetilde{H}_{k}(K)$ as groups.
In this case, we call $v$ a \emph{removable vertex} of $K$.

Let us consider the canonical action of the Weyl group $W_R$ on the vertex set $V_R$ of $K_R$.
It is known that there are exactly $n$ vertex orbits $V_1, \ldots, V_n$ of $K_R$, where $n$ is the number of simple roots of $W_R$.

\begin{theorem}\label{thm}
    For a subcomplex $L$ of $K_R$, the simplicial complex obtained by the below algorithm has the same homology group as $L$.
\end{theorem}

\begin{algorithm}[H]
\caption{~}
	\begin{algorithmic}[1]
		\State $K \leftarrow L$
		\For {$i = 1, \ldots, n$}
        \State $W \leftarrow \emptyset$
		\For {each $v \in V_i$}
        \If {$v$ is removable in $K$}
        \State {$W \leftarrow  W \cup \{ v\}$}
		\EndIf
		\EndFor
		\State {$K \leftarrow K-W := \{\sigma-W \mid \sigma \in K\}$}
		\EndFor
        \State {Return $K$}
	\end{algorithmic}
\end{algorithm}

\begin{proof}
By Proposition~8.29 in \cite{Hall_book}, for each facet $\cC$ of $K_R$, every vertex orbit of $K_R$ contains exactly one vertex of $\cC$.
That is, for any $v, w \in V_i$, $v$ and $w$ are not adjacent.
Then, for any subcomplex $K$ of $K_R$ and $v, w \in V_i$, $v$ is not contained in $Lk_{K}(w)$.

Note that for removable vertices $v$ and $w$ of $K$, $w$ is still removable in $K-v$ if $w$ is not in the link of $v$ in $K$, whereas there is no guarantee that $w$ is removable in $K-v$ in general.
Thus, we can remove all removable vertices of $K$ in $V_i$ from $K$ at once without changing their homology groups.
We do this procedure inductively for every vertex orbit to obtain $K$, and it is obvious that $H_\ast(K) \cong H_\ast(L)$ as groups.
\end{proof}

Notice that the order of vertex orbits does not matter.
In this paper, we fix the order by size of orbit, with $\vert V_{i} \vert < \vert V_{i+1} \vert$.
Let $\widehat{K}_S$ be the complex resulting from $K_S$ as obtained by the algorithm in Theorem~\ref{thm}.
Then the sizes of $\widehat{K}_S$ obtained as in Table~\ref{number of vtx} are dramatically smaller than the sizes of $K_S$.

\begin{table}[H]
	\renewcommand{\arraystretch}{1.3}
	\centering
	\begin{tabular}{c|c|c|c}
		\hline
		$E_7$ & $S = S_1$ & $S = S_2$ & $S = S_3$ \\
		\hline
		$K_S$ &$9{,}176$&$8{,}672$&$4{,}664$\\
		\hline
		$\widehat{K}_S$  & $408$ & $928$ & $4{,}664$  \\
		\hline
	\end{tabular}
	\hfil
	\renewcommand{\arraystretch}{1.3}
	\centering
	\begin{tabular}{c|c|c}
		\hline
		$E_8$ & $S = S_4$ & $S = S_5$ \\
		\hline
		$K_S$ &$432{,}944$&$451{,}200$\\
		\hline
		$\widehat{K}_S$  & $9{,}328$ & $15{,}488$ \\
		\hline
	\end{tabular}
	\caption{Numbers of vertices of $K_S$ and $\widehat{K}_S$}
	\label{number of vtx}
\end{table}

The following proposition establishes some properties of $K_S$ and $\widehat{K}_S$.
\begin{proposition}\label{pro3} \noindent
	\begin{enumerate}
        \item $K_{S_1}$ and $K_{S_4}$ have two connected components; the other $K_S$ are connected.
		\item For $S = S_1,S_4$, two components of $K_S$ are isomorphic.
		\item All $\widehat{K}_S$ are pure simplicial complexes.
		\item Each component of $\widehat{K}_{S_1}$ is isomorphic to some induced subcomplex of $K_{D_6}$.
		\item Each component of $\widehat{K}_{S_4}$ is isomorphic to $\widehat{K}_{S_3}$.
	\end{enumerate}
\end{proposition}

The above proposition was checked by computer program.
The Python codes used for checking are released at \url{https://github.com/Seonghyeon-Yu/E7-and-E8}.

In conclusion, by Proposition~\ref{pro3}, for our purposes we only need to compute the Betti numbers of $K_{S}$ for $S = S_2, S_3$, and $S_5$, since the Betti numbers of $K_S$ of $K_{D_6}$ are already computed in \cite{Choi-Kaji-Park2019} for all $S\in Row(\Lambda_{D_6})$.

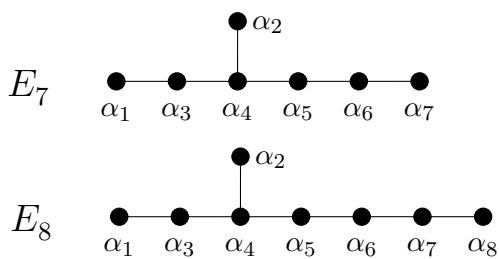
\begin{remark} \noindent
\begin{center}
    {\Large \dynkinName[Coxeter]E7\;\;\;}
    \dynkin[labels={\alpha_1,\alpha_...,\alpha_7}, text style/.style={scale=1}, scale = 0.08cm]E7 \;\;\;\;\;\;\;\\
    {\Large \dynkinName[Coxeter]E8\;\;\;}
    \dynkin[labels={\alpha_1,\alpha_...,\alpha_8}, text style/.style={scale=1}, scale = 0.08cm]E8\\
\end{center}
    \begin{itemize}
    \setlength\itemsep{.5em}
      \item[(1)] Each isomorphism in Proposition \ref{pro3} $(2)$ can be represented as one of simple roots; see the above Dynkin diagrams. For the type $E_7$, the simple root $\alpha_3$ represents the isomorphism between the components of $\widehat{K}_{S_1}$; for the type $E_8$, the simple root $\alpha_2$ represents the isomorphism between the components of $\widehat{K}_{S_4}$.
      \item[(2)] Denote by $\bar{K}_S$ a connected component of $\widehat{K}_S$.
      Here are the $f$-vectors $f(\bar{K}_{S})$ of $\bar{K}_S$:
          \begin{align*}
            f(\bar{K}_{S_1}) &=  (204, 1312, 1920)& f(\bar{K}_{S_4}) &= (4664,36288,60480) \\
            f(\bar{K}_{S_2}) &= (928,6848,15360,11520) & f(\bar{K}_{S_5}) &= (15488,193536,645120) \\
            f(\bar{K}_{S_3}) &= (4664,36288,60480) &
          \end{align*}
      As seen, the $f$-vectors of $\bar{K}_{S_3}$ and $\bar{K}_{S_4}$ are the same because of Proposition~\ref{pro3} $(5)$.
      From the $f$-vectors, we can compute the Euler characteristic of $K_S$.
    \end{itemize}
\end{remark}

\section{Computation of the Betti numbers}
In this section, we shall use a computer program \emph{SageMath 9.3} \cite{sage93}, to compute the Betti numbers of the given simplicial complexes.
From Proposition~\ref{pro3}, we already know the Betti numbers of $\widehat{K}_{S_1}$.
For $S_2$ and $S_3$, we can compute the Betti numbers of $\widehat{K}_S$ within reasonable time; see Table \ref{E7_Ks}.

\begin{table}[H]
	\renewcommand{\arraystretch}{1.3}
	\centering
	\begin{tabular}{c|c|c|c}
		\hline
		$\widetilde{\beta}_{k}(K_S)$ & $S = S_1$ & $S = S_2$ & $S = S_3$ \\
		\hline
		$k = 0$  & $1$ & $0$ & $0$  \\
		\hline
		$k = 1$  & $0$ & $129$ & $0$ \\
		\hline
		$k = 2$ & $1{,}622$ & $0$ & $28{,}855$ \\
		\hline
		$k = 3$ & $0$ & $1{,}952$ & $0$ \\
		\hline
		\hline
		$\#$ orbit &  $63$ & $63$ & $1$ \\
		\hline
	\end{tabular}
	\caption{Nonzero reduced Betti numbers of $K_S$ for $S$ in Row($\Lambda_{E_7}$)}
	\label{E7_Ks}
\end{table}

From Table~\ref{E7_Ks}, we can immediately conclude the following theorem.

\begin{theorem}
The $k$th Betti numbers $\beta_{k}$ of $X^{\R}_{E_7}$ are as follows:
    $$\beta_{k}(X^{\R}_{E_7})= \begin{cases}
1, & \mbox{if} \ k=0\\
63, & \mbox{if} \ k=1\\
8{,}127, & \mbox{if} \ k=2\\
131{,}041, & \mbox{if} \ k=3\\
122{,}976, & \mbox{if} \ k=4\\
0, & \mbox{otherwise.}
\end{cases} $$
\end{theorem}

By Proposition~\ref{pro3} and the above result, we now have the Betti numbers of $\widehat{K}_{S_{4}}$.
For any vertex $v$ of $\widehat{K}_{S_{5}}$, we have $\widetilde{H}_{0}(Lk_{\widehat{K}_{S_{5}}}(v)) = \widetilde{H}_{1}(Lk_{\widehat{K}_{S_{5}}}(v)) = 0$ by computation.
Hence we have the Mayer-Vietoris sequence
$$
0 = \widetilde{H}_{1}(Lk_{\widehat{K}_{S_{5}}}(v)) \rightarrow \widetilde{H}_{1}(\widehat{K}_{S_{5}}-v)\oplus \widetilde{H}_{1}(St_{\widehat{K}_{S_{5}}}(v)) \rightarrow \widetilde{H}_{1}(\widehat{K}_{S_{5}}) \rightarrow \widetilde{H}_{0}(Lk_{\widehat{K}_{S_{5}}}(v)) = 0.
$$
Since $\widetilde{H}_{1}(St_{\widehat{K}_{S_{5}}}(v))$ is trivial, $\widetilde{H}_{1}(\widehat{K}_{S_{5}}-v)$ is isomorphic to $\widetilde{H}_{1}(\widehat{K}_{S_{5}})$.
For the largest vertex orbit $V$ of $\widehat{K}_{S_{5}}$, by the same proof argument as for Theorem~\ref{thm}, $\widetilde{H}_{1}(\widehat{K}_{S_{5}}-{V})$ is isomorphic to $\widetilde{H}_{1}(\widehat{K}_{S_{5}})$.
Note that the size of $\widehat{K}_{S_{5}}-{V}$ is much smaller than $\widehat{K}_{S_{5}}$.
Thus, $\widetilde\beta_{1}(K_{S_{5}})$ can be computed within reasonable time from $\widehat{K}_{S_{5}}-V$ instead of $\widehat{K}_{S_{5}}$.
However, there is no vertex of $\widehat{K}_{S_{5}}$ such that $\widetilde{H}_{2}(Lk_{\widehat{K}_{S_{5}}}(v)) = 0$.
Thus, for $k=2,3$ we must compute $\widetilde\beta_{k}(\widehat{K}_{S_{5}})$ directly, which takes a few days of run time.
See Table \ref{E8_Ks} for the results.

\begin{table}[h!]
	\renewcommand{\arraystretch}{1.3}
	\centering
	\begin{tabular}{c|c|c}
		\hline
		$\widetilde{\beta}_{k}(K_S)$ & $S = S_1$ & $S = S_2$ \\
		\hline
		$k = 0$  & $1$ & $0$ \\
		\hline
		$k = 1$  & $0$ & $769$ \\
		\hline
		$k = 2$ & $57{,}710$ & $0$ \\
		\hline
		$k = 3$ & $0$ & $177{,}280$ \\
		\hline
        \hline
        $\#$ orbit & 120 & 135 \\
        \hline
	\end{tabular}
	\caption{Nonzero reduced Betti numbers of $K_S$ for $S$ in Row($\Lambda_{E_8}$)}
	\label{E8_Ks}
\end{table}

Table~\ref{E8_Ks} implies the following theorem.

\begin{theorem}
	The $k$th Betti numbers $\beta_{k}$ of $X^{\R}_{E_8}$ are as follows:
	$$\beta_{k}(X^{\R}_{E_8})= \begin{cases}
		1, & \mbox{if} \ k=0\\
		120, & \mbox{if} \ k=1\\
		103{,}815, & \mbox{if} \ k=2\\
		6{,}925{,}200, & \mbox{if} \ k=3\\
		23{,}932{,}800, & \mbox{if} \ k=4\\
		0, & \mbox{otherwise.}
	\end{cases} $$
\end{theorem}

The Euler characteristic number $\chi(X)$ of a topological space $X$ is equal to the alternating sum of the Betti numbers $\beta_{k}(X)$ of $X$.
We can use this fact as a confidence check for our results.
\begin{remark}
    It is well known that the Euler characteristic numbers $\chi(X_{E_7}^{\R}) $ and $\chi(X_{E_8}^{\R})$ are 0 and 17,111,296, respectively. Obviously, the alternating sums of the Betti numbers based on our results match $\chi(X_{E_7}^{\R})$ and $\chi(X_{E_8}^{\R})$.
\end{remark}

\bibliographystyle{plain}

\end{document}